\documentclass[12pt]{amsart}
\makeatletter

\@addtoreset{equation}{section}
\makeatother
\usepackage{amsmath,amssymb,amsthm,color}
\usepackage[T1]{fontenc}
\theoremstyle{definition}
\newtheorem{theorem}[equation]{Theorem}
\newtheorem{remark}[equation]{Remark} 
\newtheorem{corollary}[equation]{Corollary}
\newtheorem{lemma}[equation]{Lemma}
\begin{document}

\title[Polynomial-factorial Diophantine equations]{On the finiteness of solutions for polynomial-factorial Diophantine equations}
\author{Wataru Takeda}
\address{Department of Mathematics, Nagoya University,
Chikusa-ku, Nagoya 464-8602, Japan.}
\email{d18002r@math.nagoya-u.ac.jp}

\begin{abstract}
We study the Diophantine equations obtained by equating a polynomial and the factorial function, and prove the finiteness of integer solutions under certain conditions. For example, we show that there exists only finitely many $l$ such that $l!$ is represented {by} $N_A(x)$, where $N_A$ is a norm form constructed from the field norm of a field extension $K/\mathbf Q$.
We also deal with the equation $N_A(x)=l!_S$, where $l!_S$ is the Bhargava factorial. In this paper, we also show that the Oesterl\'e-Masser conjecture implies that for any infinite subset $S$ of $\mathbf Z$ and for any polynomial $P(x)\in\mathbf Z[x]$ of degree $2$ or more the equation $P(x)=l!_S$ has only finitely many solutions $(x,l)$. For some special infinite subsets $S$ of $\mathbf Z$, we can show the finiteness of solutions for the equation $P(x)=l!_S$ unconditionally.
\end{abstract}

\maketitle

\section{Introduction}
The Brocard-Ramanujan problem, which is an unsolved problem in number theory, is to find integer solutions $(x,l)$ of $x^2-1=l!$. Brocard  and Ramanujan {independently} considered this problem and conjectured that the only solutions are $(x,l)=(5,4), (11,5)$ and $(71,7)$ \cite{br1,br2,ra}. As one of its generalizations, it {has been} proposed that there are only finitely many solutions of {the} polynomial-factorial Diophantine equation \begin{equation}\label{pfe}P(x)=l!,\end{equation} where $P(x)$ is a polynomial of degree $2$ or more with integer coefficients.
{The} generalized Brocard-Ramanujan problem excludes the case $\deg P=1$. In this case, we can observe that if $a_1|a_0$ the equation $a_1x+a_0=l!$ has infinitely many solutions $(x,l)$, and otherwise has only finitely many solutions.

The Oesterl\'e-Masser conjecture, {also known as} the ABC-conjecture, implies that polynomial-factorial equations (\ref{pfe}) have only finitely many solutions. To explain the statement of the Oesterl\'e-Masser conjecture, we define the algebraic radical. For any non-zero integer $n$, the algebraic radical $\mathrm{rad}(n)$ is defined by \[\mathrm{rad}(n) =\prod_{p|n} p.\]
The Oesterl\'e-Masser conjecture states that
for any $\varepsilon>0$ there exists a positive constant $\beta(\varepsilon)$ such that \[\max\{|a|,|b|,|c|\}<\beta(\varepsilon)\mathrm{rad}(abc)^{1+\varepsilon}\] for any triples $(a,b,c)$ of non-zero coprime integers with $a+b=c$ \cite{ma,os}. In the following, we summarize {results applying this conjecture
to polynomial-factorial Diophantine equations}.

First, Overholt showed that if the weak form of Szpiro's conjecture is true, then $x^2-1=l!$ has only finitely many solutions \cite{ov93}. The weak form of Szpiro's conjecture states that there exists a positive constant ${s>6}$ such that \[|abc|<\mathrm{rad}(abc)^{s}\] for any triples $(a,b,c)$ of non-zero coprime integers with $a+b=c$. One can check that the Oesterl\'e-Masser conjecture implies this conjecture. More generally, D\k abrowski showed that if the weak form of Szpiro's conjecture is true, then {for any integer $A$ the equation} $x^2-A=l!$ has only finitely many solutions $(x,l)$ \cite{da96}. He also showed that when $A$ is not square of an
integer this result becomes unconditional.
As a generalization of these results, Luca showed that for {any} polynomial $P(x)$ with integer coefficients of degree $\ge2$, the Oesterl\'e-Masser conjecture implies that the equation $P(x)=l!$ has only finitely many solutions $(x,l)$ \cite{lu02}.

There are many unconditional results. It is known that for $m\ge3$ the equation $x^m+ y^m=l!$ has no solution with $\gcd(x,y)=1$ except for $(x,y,l)=(1,1,2)$ and $x^m-y^m=l!$ has no solution with $\gcd(x,y)=1$ except for $m=4$ \cite{eo37}. In 1973, Pollack and Shapiro showed that $x^4-1=l!$ also has no solution \cite{ps73}. If we remove the assumption $\gcd(x,y)=1$, there are infinitely many solutions $(x,y,l)$ of $x^2-y^2=l!$. In fact, for any $a\ge4$, $(x,y,l)=(\frac{a!}4+1,\frac{a!}4-1,a)$ are solutions of $x^2-y^2=l!$.

D\k abrowski and Ulas showed that for all positive integer $B$ there exist infinitely many integers $A$ such that $x^2-A=B\cdot l!$ has at least three solutions in positive integer $x$ \cite{du}.

Berend and Osgood {dealt with all polynomial $P(x)\in\mathbf Z[x]$ and} showed that for any polynomial $P$ of degree $2$ or more with integer coefficients, the equation $P(x) = l!$ has only a density $0$ set of solutions $l$ \cite{bo92}, that is, \[\lim_{n\rightarrow \infty}\frac{|\{l\le n~|~\text{there exists $x\in\mathbf Z$ such that $P(x)=l!$}\}|}n=0.\]
In 2006, Berend and Harmse considered several related problems. They showed that for any polynomial $P$ which is an irreducible polynomial or satisfies {certain technical
conditions}, there exist only finitely many solutions of $P(x)=H_l$, where $H_l$ is a highly divisible sequence \cite{bh}.
They chose the following three sequences as highly divisible sequences $H_l$,
\begin{itemize}
\item$H_l = l!$.
\item $H_l = [1, 2,\ldots,l]$,
\item[]where $[1, 2,\ldots,l]$ is the least common multiple of all positive integers less than or equal to $l$.
\item $H_l = p_1p_2 \cdots p_l$,
\item[] where $2=p_1 < p_2 <\cdots<p_l<\cdots$ is the sequence of all primes.
\end{itemize}
They also consider $H_l$ {defined by} multinomial coefficients \[\binom{al}{l,\ldots,l}=\frac{(al)!}{(l!)^a}\]for a fixed integer $a$.

In this paper, we use algebraic number theoretical approaches to consider the number of solutions for equation (\ref{pfe}). Therefore, it is appropriate for our approach to deal with the equation
\[
\sum_{i=0}^na_ix^iy^{n-i}=l!.
\]
More generally, we focus on the equation
\begin{equation}\label{hpfe}
\sum_{i=0}^na_ix^iy^{n-i}=\Pi_K(l),
\end{equation}
where $a_i\in\mathbf Z$ and $\Pi_K$ is {a} generalized factorial function over number fields. Let $K$ be a number field and $\mathcal O_K$ be its ring of integers. Then the function $\Pi_K(l)$ is defined by \[\Pi_K(l)=\prod_{\substack{\mathfrak a:\text{ideal}\\\mathfrak {Na}\le l}}\mathfrak{Na},\]
where $\mathfrak{Na}=\#\mathcal O_K/\mathfrak a$.

Instead of studying the number of $(x,y,l)$, we study the number of $l$ {for} which there exists a pair $(x,y)$ such that $a_nx^n+\cdots+a_0y^{n}=\Pi_K(l)$ {by} the following reasons. It is known that when $d$ is not a square integer $x^2-dy^2=1$ has infinitely many solutions from {the} theory of Pell's equation. Therefore, we can find $x^2-dy^2=l!$ has infinitely many solutions $(x,y,l)$ easily. To consider the relation between integers represented {by} polynomial and those of factorial, we consider the number of $l$ for which there exists a pair $(x,y)$ such that $a_nx^n+\cdots+a_0y^{n}=\Pi_K(l)$.
In the case $K=\mathbf Q$ and $y=1$, equation (\ref{hpfe}) is reduced to {the} generalized Brocard-Ramanujan's equation (\ref{pfe}).
Therefore, it is expected that our results give some improvement of {the} generalized Brocard-Ramanujan problem. For the equation $a_nx^n+\cdots+a_0y^{n}=l!$, we show that
\begin{theorem}
Let $F(x,y)$ be a homogeneous irreducible polynomial with $\deg F\ge2$, then there exist only finitely many $l$ such that $l!$ is represented {by} $F(x,y)$.
\end{theorem}
As a corollary of this theorem, we obtain {the result of} Berend and Harmse \cite{bh} for irreducible polynomials. For reducible polynomials, we prove the following result.
\begin{theorem}
\label{bbb}
Let $F(x,y)$ a polynomial in $\mathbf Z[x,y]$ whose irreducible factorization is \[F(x,y)=\prod_{j=1}^uF_j(x,y).\]
Assume that there exist positive integers $a$ and $b>1$ such that \[\bigcap_{i} \mathbf P(\mathcal C_{F_i})\supset\{p:\text{prime}~|~p\equiv a\mod b\}.\]
Then there exist only finitely many $l$ such that $l!$ is represented {by} $F(x,y)$.
\end{theorem}
We explain the definition of $\mathbf P(\mathcal C_{F_i})$ in Section 4, roughly speaking, $\mathbf P(\mathcal C_{F_i})$ is the set of all primes $p$ such that $F_i(x,1)\equiv0\mod p$ has no solution in $\mathbf Z/p\mathbf Z$.

Taking $y=1$ in Theorem \ref{bbb}, we get the result of Berend and Harmse for reducible polynomials partially.
As {in} their results, we can replace $l!$ with the above highly divisible sequences $H_l$.
Theorem \ref{bbb} is one of corollaries of the  folloing result for equation (\ref{hpfe}).
\begin{theorem}
\label{main2}
Let $K$ be a Galois extension field over $\mathbf Q$ and $F(x,y)$ a polynomial in $\mathbf Z[x,y]$ whose irreducible factorization is \[F(x,y)=\prod_{j=1}^uF_j(x,y)^{m_j}.\]
Assume that there exist a conjugacy class $C$ of ${\rm Gal}(K/\mathbf Q)$, positive integers $a$ and $b>1$ such that $\mathbf P(C)\cap \bigcap_{i} \mathbf P(\mathcal C_{F_i})\supset\{p:\text{prime}~|~p\equiv a\mod b\}$.
If $\deg F$ does not divide $[K:\mathbf Q]$ then there exist only finitely many $l$ such that $\Pi_K(l)$ is represented {by} $F(x,y)$.
\end{theorem}
Moreover, for special quadratic forms $F(x,y)$ there exist infinitely many $l$ with $\Pi_K(l)$ represented {by} $F$.

In Section 2, we review some of the standard facts on algebraic number theory. First, we check that for any Galois group $G\not=\{1\}$ of splitting field of polynomial $P$, there exists an element such that it fixes no roots of $P$. This fact is very important to prove our theorems. Second, we review the Frobenius map of prime $p$. This map controls the irreducible factorization modulo $p$ of the polynomial corresponding to itself.

In Section 3, we introduce two auxiliary lemmas. The first one characterizes the prime factorization of integers which can be written as $F(x,y)$.
This is one of the generalizations of Cho's results \cite{ch14,ch16}. Cho characterized the prime factorization of integers written as $x^2+ny^2$ or $x^2+xy+ny^2$ with congruent condition for $n\ge1$.
The second one gives a Bertrand type estimate for prime ideals corresponding to a conjugacy class of Galois group such that their ideal norm is of the form $p^f$.
In a previous paper \cite{ta18}, we considered a Bertrand type estimate for primes splitting completely, which is the simplest case of the second auxiliary lemma.
This result plays a crucial role in proving our main theorem.

In Section 4, we give a necessary condition for the existence of infinitely many solutions of $F(x,y)=\Pi_K(l)$. As a corollary, we obtain the result of Berend and Harmse for irreducible polynomials. Also, we gave their result for reducible polynomials partially.
Moreover, for special quadratic forms $F(x,y)$ there exists infinitely many $l$ such that $\Pi_K(l)$ is represented {by} $F$.

In Section 5, we consider a generalization of the result in Section 4 to multi-variable homogeneous polynomials. However, irreducible polynomial $x^2+y^2+z^2+w^2$ represents any positive integers, that is, there exists infinitely many $l$ such that $H_l$ is represented {by} $x^2+y^2+z^2+w^2$. Therefore, the simplest generalization of the result of Section 4 does not hold. In place of multi-variable homogeneous polynomials, we deal with norm forms of number fields. We can apply the same argument {as in} the proof of the result in Section 4 and obtain the finiteness of $l$ such that $H_l$ is represented {by} a norm form $N_{\alpha_1,\ldots,\alpha_n}(x_1,\ldots,x_n)$ except for the case $n=1$. Since $N_{\alpha_1}(x)=x$, the equation $N_{\alpha_1}(x)=H_l$ has infinitely many solutions $(x,l)$. After that, we consider the equation $N_{\alpha_1,\ldots,\alpha_n}(x_1,\ldots,x_n)=l!_S$, where $l!_S$ is the Bhargava factorial for $S\subset \mathbf Z$. Since the Bhargava factorial $l!_{\mathbf Z}$ is the ordinary factorial $l!$, we can regard this equation as one of the generalizations of the Brocard-Ramanujan problem. We point out that we can generalize Luca's result, which states that the Oesterl\'e-Masser conjecture implies the finiteness of solution of the equation $P(x)=l!$, to the equation $P(x)=l!_S$ by following Luca's proof. Also, we show the finiteness of solution of the equation $N_{\alpha_1,\ldots,\alpha_n}(x_1,\ldots,x_n)=l!_{S(a,b,c)}$ for $S(a,b,c)=\{an^2+bn+c~|~n\in\mathbf Z\}$ with $(a,b,c)\in\mathbf{Z}^3$ and $ab\neq0$.


\section{Preliminaries on algebraic number theory}
We recall some basic definitions and some propositions of algebraic number theory in this section. Let $P(x)=a_nx^{n}+\cdots+a_0\in\mathbf Z[x]$ be an irreducible polynomial with $\deg P=n$ and discriminant $\Delta_P$. When $\alpha_1,\ldots,\alpha_n$ are roots of $P$, the splitting field $K_P$ of $P$ is $\mathbf Q(\alpha_1,\ldots,\alpha_n)$. It is known that $K_P/\mathbf Q$ is a Galois extension and the Galois closure of $\mathbf Q(\alpha_i)/\mathbf Q$ for all $i=1,\ldots,n$. The Galois group $G_P$ of $K_P/\mathbf Q$ plays a crucial role for splitting of primes in $\mathbf Q(\alpha_i)$ as follows.

We call a subgroup $H$ of $S_n$ transitive if the group orbit $H(i)=\{\sigma(i)~|~\sigma\in H\}$ is equal to $\{1,\ldots,n\}$ for $1\le i\le n$.
From Galois Theory, {the} Galois group $G_P$ can be identified with a transitive subgroup $H$ of the symmetric group $S_n$ of degree $n$. For $\sigma\in H$, the cycle type of $\sigma$ is defined as the ascending ordered list $[f_1,\ldots,f_r]$ of the sizes of the cycles in the cycle decomposition of $\sigma$. For example, the cycle type of $(1\ 2)(3\ 4)(6\ 7\ 8)=(1\ 2)(3\ 4)(5)(6\ 7\ 8)(9)\in H\subset S_{9}$ is $[1,1,2,2,3]$. Since if two permutations are conjugate in $H$ then they have the same cycle type, we can define the cycle type of conjugacy class $C=[\sigma]$ of $H$ by the cycle type of a representative $\sigma$. We introduce a lemma for transitive subgroups of the symmetric group $S_n$.
\begin{lemma}
\label{class}
Let $H$ be a transitive subgroup of the symmetric group $S_n$ of degree $n\ge2$. Then there exists an element $\sigma\in H$ such that $\sigma(i)\not=i$ for all $i=1,\ldots,n$.
\end{lemma}
\begin{proof}
Let $H_i$ be the stabilizer subgroup of $H$ with respect to $i$. Then the orbit-stabilizer theorem and transitivity of $H$ leads to $|H_i|=|H|/n$. Now we consider the number of elements of the set \[S=\{\sigma~|~\text{there exists $i$ such that $\sigma(i)=i$}\}.\] Since the identity element belongs to all stabilizer $H_i$, we have
\[|S|\le\sum_{i=1}^n |H_i|-(n-1)=|H|-n+1<|H|.\]
Therefore, there exists an element $\sigma\in H$ such that $\sigma(i)\not=i$ for all $i=1,\ldots,n$.
\end{proof}
This lemma ensures that for any Galois group $G\not=\{1\}$ of an irreducible polynomial of degree $\ge2$ contains an element that does not fix any of the roots.

Next, we review the definitions and properties of the Frobenius map.
Let $p$ be a prime and $\mathfrak P$ {a} prime ideal of $\mathcal O_{K_P}$ lying above $p$. For {a} prime ideal $\mathfrak P$ in $\mathcal O_{K_P}$, we define the decomposition group $D_{\mathfrak P}$ of $\mathfrak P$ by $\{\sigma\in G_P~|~\sigma(\mathfrak P)=\mathfrak P\}$. Since $\sigma(\mathfrak P)=\mathfrak P$ and $\sigma(\mathcal O_{K_P})=\mathcal O_{K_P}$ for $\sigma\in D_{\mathfrak P}$,
$\sigma$ induces an automorphism $\overline{\sigma}$ of $\mathcal O_{K_P}/\mathfrak P$ over $\mathbf Z/p\mathbf Z$. Now we consider the Galois group ${\rm Gal}((\mathcal O_{K_P}/\mathfrak P)/(\mathbf Z/p\mathbf Z))$.  It is known that this group is cyclic and there exists an unique automorphism $\sigma: x\rightarrow x^{p}$ which generates it.
Then the Frobenius map $(p,K_P/\mathbf Q)$ of $p$ is the image of $\sigma$ in Galois group $G_P$. If the Frobenius map $(p,K_P/\mathbf Q)$ of $p$ belongs to a conjugacy class $C$ of $G_P$, then we say that $p$ corresponds to $C$.
We denote the set of primes corresponding to $C\in\mathcal C$ by $\mathbf P(\mathcal C)$. The following theorem gives a relation between the cycle type of {the} Frobenius map of $p$ and the monic irreducible factorization of $P(x) \mod p$, where $p$ does not divide $a_n\Delta_P$.
\begin{theorem}[Frobenius]
\label{frob}
Let $p$ be a prime such that $p$ does not divide $a_n\Delta_P$. We denote the cycle type of the Frobenius map $(p,K_P/\mathbf Q)$ of $p$ by $[f_1,\ldots,f_r]$. Then the monic irreducible factorization of $P(x) \mod p$ is $P(x)\equiv a_nP_1(x)\cdots P_ r(x)\mod p$,
where $P_i(x)$ are distinct and $f_i = \deg P_i(x)$.
\end{theorem}
\section{Auxiliary lemmas}
Let $F(x,y)=a_nx^{n}+a_{n-1}x^{n-1}y+\cdots+a_0y^{n}$ be an irreducible homogeneous polynomial {(i.e. binary form)} and $K_F$ the splitting field of $F(x,1)$. Also, we define the modified discriminant $\Delta_{F,mod}$ by \[\Delta_{F,mod}=\frac {\Delta_{F}}{\gcd(a_n,\ldots,a_0)^{2n-2}},\]
where $\Delta_F$ is the discriminant of $F(x,1)$. 
Cho {studied} the representation of integers as $x^2+ny^2$ or $x^2+xy+ny^2$ for $n\ge1$ under the congruence conditions $x\equiv1\mod m$ and $y\equiv 0\mod m$ \cite{ch14,ch16}. We remove this congruence conditions and consider all polynomials with integer coefficients.  Let $\mathcal C_F$ be the set of conjugacy classes $C$ of the Galois group $G_F={\rm Gal}(K_F/\mathbf Q)$ whose cycle type $[f_1,\ldots,f_r]$ satisfies $f_i\ge2$ for all $i=1,\ldots,r$. This classification is very important to characterize integers represented {by} $F(x,y)$.

\begin{lemma}
\label{123}
Let $F(x,y)=a_nx^{n}+a_{n-1}x^{n-1}y+\cdots+a_0y^{n}$ be an irreducible homogeneous polynomial with $g=\gcd(a_n,a_{n-1},\ldots,a_0)$.
Let $N$ be an integer with \[N=gp_1\cdots p_sq_1^{l_1}\cdots q_t^{l_t},\]
where $q_i$ are primes corresponding to a conjugacy class $C\in\mathcal C_F$ with\\ $\gcd(q_i,a_n\Delta_{F,mod})=1$ or $\gcd(q_i,a_0\Delta_{F,mod})=1$ and $p_i$ are the other primes.
If $N$ is represented {by} $F(x,y)$ then $n|l_i$ for all $i$.
\end{lemma}
\begin{proof}
An integer $N$ is represented {by} $a_nx^{n}+a_{n-1}x^{n-1}+\cdots+a_0y^{n}$ with $\gcd(a_n,\ldots,a_0)=g$ if and only if $\frac Ng$ is represented {by} $\frac{a_n}gx^{n}+\cdots+\frac{a_0}gy^{n}$ with $\gcd(\frac{a_n}g,\ldots,\frac{a_0}g)=1$. Hence, we may assume without loss of generality that $\gcd(a_n,a_{n-1},\ldots,a_0)=1$ and $\Delta_{F}=\Delta_{F,mod}$ .

Let $q$ be a prime corresponding to a conjugacy class $C\in\mathcal C_F$ with\\ $\gcd(q_i,$ $a\Delta_{F,mod})=1$, where $a\in\{a_n,a_0\}$. By the definition of $\mathcal C_F$, the cycle type $[f_1,\ldots,f_r]$ of $C$ satisfies $f_i\ge2$ for all $i=1,\ldots,r$. As we seen in Lemma \ref{frob}, the monic irreducible factorization of $F(x,y)\mod q$ is \[aF_1(x,y)\cdots F_ r(x,y)\mod q,\]
where the $F_i(x,y)$ are distinct and $f_i = \deg F_i(x,y)$. Since $\deg F_i(x,y)\ge2$ for all $i=1,\ldots,r$, it {follows} that $F(x,y) \mod q$ has no solution expect for $(0,0)$ in $\mathbf F_q\times\mathbf F_q$. Therefore, if $q|N$ then $q$ divides both of $x$ and $y$. Thus $q^n|N$. {We can apply} the same argument repeatedly {to $F(\frac x{q^n},\frac y{q^n})=\frac N{q^{n}}$, which} leads to $n|l_i$.

This proves the lemma.
\end{proof}
As one of the corollaries of Lemma \ref{123}, we obtain the following theorem.

\begin{theorem}
\label{12}
Let $F(x,y)=a_nx^{n}+a_{n-1}x^{n-1}y+\cdots+a_0y^{n}$ be a homogeneous polynomial whose irreducible factorization is \[F(x,y)=\prod_{j=1}^uF_j(x,y)\] and $g=\gcd(a_n,a_{n-1},\ldots,a_0)$.
Let $N$ be an integer with \[N=gp_1\cdots p_sq_1^{l_1}\cdots q_t^{l_t},\]
where $q_i\in\cap_{j=1}^u\mathcal C_{F_j}$ with $\gcd(q_i,a_n\Delta_{F,mod})=1$ or $\gcd(q_i,a_0\Delta_{F,mod})=1$ and $p_i$ are the other primes.
If $N$ is represented {by} $F(x,y)$ then $n|l_i$ for all $i$.
\end{theorem}
\begin{proof}
By assumption, there exists a pair $(x_0,y_0)\in\mathbf Z^2$ such that $N=F(x_0,y_0)$. Since $q_i|N=F(x_0,y_0)$, there exists a polynomial $F_j(x,y)=a_{j,n_j}x^{n_j}+\cdots+a_{j,0}y^{n_j}$ such that $q_i|F_j(x_0,y_0)$. By the definition of discriminant of polynomial, we have $\Delta_{F_j,mod}|\Delta_{F,mod}$, that is, $\gcd(q_i,a\Delta_{F,mod})=1$ implies $\gcd(q_i,a^{(j)}\Delta_{F_j,mod})=1$, where $a\in\{a_n,a_0\}$ and $a^{(j)}\in\{a_{j,n_j},a_{j,0}\}$. It {follows} that $x_0\equiv y_0\equiv0\mod q$ by Lemma \ref{123}. Therefore we obtain $q^n|N=F(x_0,y_0)$ and $n|l_i$. This is the desired conclusion.
\end{proof}
Next we change the assumption of Lemma \ref{123} and give a necessary and sufficient condition for integers {to be} represented {by} $F(x,y)$. We call a discriminant $\Delta$ fundamental, if one of the following statements holds
\begin{itemize}
\item $\Delta\equiv1\mod 4$ and is square-free,
\item $\Delta=4m$, where $m\equiv2$ or $3\mod 4$ and $m$ is square-free.
\end{itemize}
In the following, we assume that one of $a$ and $c$ is a prime number or $1$ and the discriminant $\Delta_F$ is fundamental. We characterize the prime factorization of integers which are expressed {by} $ax^2+bxy+cy^2$.

\begin{theorem}
\label{1234}
Let $F(x,y)=ax^2+bxy+cy^2$ be a positive definite quadratic form with fundamental modified discriminant $\Delta_{F,mod}$ and $g=\gcd(a,b,c)$. We denote the corresponding order to $F$ by $\mathcal O$ and the set of principal ideals of $\mathcal O$ by $P_{\mathcal O}$. Let $N$ be an integer with \[aN=2^{l}gp_1\cdots p_sq_1\cdots q_tr_1^{l_1}\cdots r_u^{l_u},\]
where $p_i$ ramifies in $\mathbf Q(\sqrt {\Delta_{F}})$, $q_i$ splits completely in $\mathbf Q(\sqrt {\Delta_{F}})$ and $r_i$ are distinct odd inert primes in $\mathbf Q(\sqrt {\Delta_{F}})$.
If $a$ is a prime number or $1$, then $N$ is represented {by} $F(x,y)$ if and only if
\begin{enumerate}
\item[1.] $l$ is even if $\Delta_{F,mod}\equiv5\mod8$.
\item[2.] $l_i$ are even numbers.
\item[3.] There exist prime ideals $\mathfrak{p}, \mathfrak{p}_1,\ldots,\mathfrak p_s, \mathfrak q_1,\ldots,\mathfrak q_t$ lying above $2$, $p_1,\ldots, p_s,$\\ $q_1,\ldots,q_t$ respectively such that
\[\mathfrak{p}^\ell\mathfrak{p}_1\cdots\mathfrak p_s\mathfrak q_1\cdots\mathfrak q_t(r_1)^{\frac{l_1}2}\cdots(r_u)^{\frac{l_u}2}\in P_{\mathcal O},\]
where $\ell$ is $\frac l2$ if $2$ is inert in $\mathbf Q(\sqrt {\Delta_{F}})$, $l$ otherwise.
\end{enumerate}
\end{theorem}
\begin{proof}
As in the proof of Lemma \ref{123}, we can assume $g=1$ without loss of generality.
We assume that there exists a pair $(x,y)\in\mathbf Z^2$ such that $N=ax^2+bxy+cy^2$. We can obtain the second assertion by Lemma \ref{123}.
Therefore, it suffices to prove the first and the third assertion.

First, we prove the first assertion. If $\Delta_F\equiv5\mod8$, $2|aN$ and $y$ is odd, then
$\Delta_F=b^2-4ac$ is odd, that is, $b$ is. Therefore,
\begin{align*}
aN&=a^2x^{2}+abxy+acy^{2}\\
&\equiv a^2x^{2}+ax+\frac{b^2-\Delta_F}{4}\mod 2.\\
\intertext{Since $\Delta_F\equiv5\mod8$ we obtain}
aN&\equiv a^2x^{2}+ax+1\mod2\\
&\equiv 1\mod2.
\end{align*}
This contradicts $2|aN$ and leads to $2|y$. From the identity $aN=a^2x^{2}+abxy+acy^{2}$ and this result, $ax$ is even. Therefore, $4|aN$ holds. The same argument repeatedly applied to $\frac{aN}4$ leads to $2|l$.

Next we show the third assertion. Since $\Delta_F=b^2-4ac$, we have
\begin{align*}
aN&=a^2x^2+abxy+acy^2\\
&=\left(ax+\frac{b+\sqrt{\Delta_F}}2y\right)\left(ax+\frac{b-\sqrt{\Delta_F}}2y\right)\\
&=\left(ax+\frac{b-\Delta_F}2y+\frac{\Delta_F+\sqrt{\Delta_F}}2y\right)\left(ax+\frac{b-\Delta_F}2y+\frac{\Delta_F-\sqrt{\Delta_F}}2y\right).
\end{align*}
Since $b\equiv \Delta_F\mod2$, $\frac{b-\Delta_F}2$ is an integer. Therefore, $ax+\frac{b-\Delta_F}2y+\frac{\Delta_F+\sqrt{\Delta_F}}2y\in\mathcal O$ and $\left(ax+\frac{b-\Delta_F}2y+\frac{\Delta_F+\sqrt{\Delta_F}}2y\right)\mathcal O\in P_{\mathcal O}$. Now we denote this ideal by $\mathfrak b$ then we have $(aN)\mathcal O=\mathfrak {b\overline{b}}$. Since $p_i$ ramifies in $Q(\sqrt {\Delta_F})$, $p_i\mathcal O=\mathfrak p_i^2$ and $\mathfrak p_i|\mathfrak b$.
Also, $r_i\mathcal O$ is a prime ideal in $Q(\sqrt {\Delta_F})$ and $(r_i\mathcal O)^{\frac{l_i}2}|\mathfrak b$. Also, $q_i$ splits completely as $q_i\mathcal O=\mathfrak q_i\overline{\mathfrak q}_i$ in $Q(\sqrt {\Delta_F})$. Hence, there exist prime ideals $\mathfrak q_i$ such that \[\mathfrak b=\mathfrak{p}^\ell\mathfrak{p}_1\cdots\mathfrak p_s\mathfrak q_1\cdots\mathfrak q_t(r_1)^{\frac{l_1}2}\cdots(r_u)^{\frac{l_u}2}\in P_{\mathcal O}.\]

Conversely, let $\mathfrak b$ be $\mathfrak{p}_1\cdots\mathfrak p_s\mathfrak q_1\cdots\mathfrak q_t(r_1)^{\frac{l_1}2}\cdots(r_u)^{\frac{l_u}2}\in P_{\mathcal O}$. This ideal can be expressed as \[\mathfrak b=\left(x+\frac{\Delta_F+\sqrt{\Delta_F}}2y\right)\mathcal O\]
and $(aN)\mathcal O=\mathfrak {b\overline{b}}$. Since $aN>0$, we obtain $aN=x^2+\Delta_Fxy+\frac{\Delta_F^2-\Delta_F}4y^2$. {Therefore we get}
\begin{align*}
x^2+\Delta_Fxy+\frac{\Delta_F^2-b^2}4y^2&\equiv0\mod a\\
\left(x+\frac {\Delta_F+b}2y\right)\left(x+\frac {\Delta_F-b}2y\right)&\equiv0\mod a.
\end{align*}
Now we assume that $a$ is a prime number or $1$, so $a|(x+\frac {\Delta_F+b}2y)$ or $a|(x+\frac {\Delta_F-b}2y)$.
Without loss of generality, we assume that $a\alpha=x+\frac {\Delta_F-b}2y$ for some integer $\alpha$. If we take $y=\beta$, then $aN=x^2+\Delta_Fxy+\frac{\Delta_F^2-d}4y^2=a^2\alpha^2+ab\alpha\beta+ac\beta^2$.
This proves the theorem.
\end{proof}
\begin{remark}
By swapping $x$ and $y$ in the binary form $ax^2+bxy+cy^2$ , we can replace $a$ by $c$ in Theorem \ref{1234}
\end{remark}
In the following, we consider a Bertrand type estimate for primes corresponding to a conjugacy class $C$ of Galois group $G$ by following the {method} of Hulse and Murty. They gave one of the generalizations of Bertrand's postulate, or Chebyshev's theorem, to number fields \cite{pp}.
We can obtain the following lemma, which gives a Bertrand type estimate for primes corresponding to $C$ of $G$, by following the argument of Hulse and Murty \cite{pp}.
\begin{lemma}[cf. \cite{pp, ta18}]
\label{pp}
Let $L/K$ be a Galois extension with $[L:\mathbf Q]=n$ and $D_L$ the absolute value of the discriminant of $L$.
For any $A > 1$ there exists an effectively computable constant $c(A) > 0$ such that for
$x > \exp(c(A)n(\log D_L)^2)$ there is a prime corresponding to a conjugacy class $C$ of ${\rm Gal}(L/K)$ with
$p\in (x,Ax)$.
\end{lemma}
Next, we consider the distribution of prime ideals $\mathfrak p$ corresponding to a conjugacy class $C$ such that their ideal norm is of the form $p^f$.
\begin{theorem}
\label{nan}
Let $L$ be the Galois closure of $K/\mathbf Q$ with $k=[L:\mathbf Q]$ and $p$ a prime corresponding to a conjugacy class $C$ of ${\rm Gal}(L/\mathbf Q)$.
For any $A > 1$ there exists an effectively computable constant $c(A) > 0$ such that for
$p^{f_i} > \exp(c(A)k(\log D_L)^2)$ there exists a prime ideal $\mathfrak q$ with
$\mathfrak{Nq}=q^{f_i}\in(p^{f_i}, Ap^{f_i})$, where $q\in\mathbf P(C)$.
\end{theorem}
\begin{proof}
Let $\mathfrak p$ be a prime ideal with $\mathfrak{Np}=p^{f}$ and $\mathfrak P$ a prime ideal lying above $\mathfrak p$. We denote the order of the decomposition group $D_{\mathfrak P/\mathfrak p}$ by $l$, then $\mathfrak N\mathfrak P=p^{fl}$. From Lemma \ref{pp}, any $A > 1$ there exists a constant $c(A) > 0$ such that for
$p^{fl}=\mathfrak N\mathfrak P > \exp(c(A)k(\log D_L)^2)$ there is a prime $\mathfrak Q$ corresponding to $C_{\mathfrak P}$ with
$q^{fl}=\mathfrak N\mathfrak Q\in (p^{fl},Ap^{fl})$.

Since $p,q\in\mathbf P(C_{\mathfrak P})$, there exists a prime ideal $\mathfrak Q$ with $\mathfrak N\left(\mathfrak Q \cap \mathcal O_{K}\right)=q^f$. We denote $\mathfrak q=\mathfrak Q \cap \mathcal O_{K}$.
Thus for $p^{f} > \exp(c(A)k(\log D_L)^2)$ there is a prime ideal $\mathfrak q$ with
$\mathfrak{Nq}=q^{f}\in(p^{f}, Ap^{f})$.
\end{proof}


\section{Main theorems}
In this section, we give a necessary condition for the existence of infinitely many solutions for the equation $F(x,y)=\Pi_K(l)$.

First, we consider the equation $F(x,y)=l!$, where $F$ is an irreducible homogeneous polynomial.
\begin{theorem}
\label{main1}
Let $F(x,y)=a_nx^{n}+a_{n-1}x^{n-1}y+\cdots+a_0y^{n}$ be a homogeneous irreducible polynomial with $\deg F\ge2$, then there exist only finitely many $l$ such that $l!$ is represented {by} $F(x,y)$.
\end{theorem}
\begin{proof}
Lemma \ref{class} provides that $\mathcal C_F\not=\emptyset$.
Let $C\in\mathcal C_F$ be a fixed conjugacy class of $G_F$.
The assumption $\deg F\ge2$ and Lemma \ref{123} imply that if $N$ is represented {by} $F(x,y)$ and $p|N$ for prime $p$ corresponding to $C$ with $\gcd(q_i,a_n\Delta_{F,mod})=1$ or $\gcd(q_i,a_0\Delta_{F,mod})=1$, then $N$ is divisible by $p^2$ at least. In particular, $F(x,y)=p!$ has no integer solution $(x,y)$.
Moreover, since the second smallest positive integer divisible by $p$ is $2p$, $l!$ is not of the form in Lemma \ref{123} for $p\le l<2p$, that is, there exists no pair $(x,y)\in\mathbf Z^2$ such that $F(x,y)=l!$ for $p\le l<2p$.

Let $\alpha$ be a root of $F(x,1)$ and let $k$ be the extension degree of $K_F/\mathbf Q$. We denote the ring of integers of $\mathbf Q(\alpha)$ by $\mathcal{O}_{\alpha}$.
Theorem \ref{nan} states that there exists $c > 0$ such that for
$x > \exp(ck(\log D_{K_F})^2)$ there is a prime ideal $\mathfrak p$ of $\mathcal{O}_{\alpha}$ corresponding to $C$ with $\mathfrak{Np}=p^f\in (x,2x)$.
Let $\mathfrak p$ be a prime ideal of $\mathcal{O}_{\alpha}$ corresponding to $C$ with $\mathfrak{Np}=p^f>\max\{\exp(ck(\log D_{K_F})^2), (a_n\Delta_{F,mod})^f,$ $(a_0\Delta_{F,mod})^f\}$.
Since we have $p^f>\exp(ck(\log D_{K_F})^2)$, there exists $\mathfrak q$ corresponding to $C$ with $\mathfrak{Nq}=q^f\in (p^f,2p^f)$, that is, there exists a prime $q$ corresponding to $C$ with $q\in (p,2p)$.

As above, $l!$ is not of the form in Lemma \ref{123} for $q\le l<2q$ and there exists a prime $q_1$ corresponding to $C$ with $q_1\in (q,2q)$. By induction, $l!$ is not of the form in Lemma \ref{123} for $p\le l$. This shows the finiteness of $l$ such that $l!$ is represented {by} $F(x,y)$.

\end{proof}
As a corollary of this theorem, we obtain the result of Berend and Harmse for irreducible polynomial.
\begin{theorem}[Theorem 3.1. of \cite{bh}]
\label{q1}
For any irreducible polynomial $P(x)\in\mathbf Z[x]$ with $\deg P\ge2$, the equation $P(x)=H_l$ has only finitely many solutions $(x,l)$.
\end{theorem}
Next we consider the general case $F(x,y)=\Pi_K(l)$.
For a prime $p$ and its Frobenius map $(p,K_P/\mathbf Q)$ with cycle type $[f_1,\ldots,f_r]$, we define $G_p(l;K)$ as the number of $f_i$ such that $f_i=l$.  If $K/\mathbf Q$ is a Galois extension with extension degree $k$, then $fG_p(f;K)=k$ for all primes $p$ unramified in $K$, where $f$ is the inertia degree of $p$ in $K$. Therefore, we obtain the following theorem.
\begin{theorem}
\label{fin}
Let $K$ be a Galois extension of $\mathbf Q$ and $F(x,y)$ a polynomial in $\mathbf Z[x,y]$ whose irreducible factorization is \[F(x,y)=\prod_{j=1}^uF_j(x,y).\]
Assume that there exist a conjugacy class $C$ of ${\rm Gal}(K/\mathbf Q)$, positive integers $a$ and $b>1$ such that $\mathbf P(C)\cap \bigcap_{i} \mathbf P(\mathcal C_{F_i})\supset\{p:\text{prime}~|~p\equiv a\mod b\}$.
If $\deg F$ does not divide $[K:\mathbf Q]$ then there exist only finitely many $l$ such that $\Pi_K(l)$ is represented {by} $F(x,y)$.
\end{theorem}
\begin{proof}
We denote the extension degree $[K:\mathbf Q]$ by $k$. As we remarked above, $fG_p(f;K)=k$ for all primes $p$ unramified in $K$, where $f$ is the inertia degree of $p$ in $K$. Then for all primes $p$ unramified in $K$, $p^{fa(p^f)}=p^k$. Since $\deg F$ does not divide $[K:\mathbf Q]$, Lemma \ref{12} implies that there is no pair $(x,y)$ such that $F(x,y)=\Pi_K(p^{f})$ for $p\in\{p:\text{prime}~|~p\equiv a\mod b\}$.

By {a} Bertrand type estimate for primes in an arithmetic progression $bm+a$, there exists $c> 0$ such that for $x > c$ there is a prime $p\in\{p:\text{prime}~|~p\equiv a\mod b\}$ with $p\in (x,2^{\frac1{f}}x)$.
Let $\mathfrak p$ be a prime ideal of $K$ corresponding to $C$ with $\mathfrak{Np}=p^f>\max\{c^f, (a_n\Delta_{F,mod})^f, (a_0\Delta_{F,mod})^f\}$.
Since $p>c$, there exist a prime $q$ satisfying $q=bm_q+a\in(p,2^{\frac1f}p)$. By the assumption, we have $q\in\mathbf P(C)$ and there exist $\mathfrak q$ lying above $q$ with $\mathfrak{Nq}=q^f\in (p^f,2p^f)$.

As above, $\Pi_K(l)$ is not of the form in Lemma \ref{12} for $q^f\le l<2q^f$ and there exists a prime ideal $\mathfrak q_1$ corresponding to $C$ with $\mathfrak{Nq}_1\in (q^f,2q^f)$. By induction, $\Pi_K(l)$ is not of the form in Lemma \ref{12} for $p^f\le l$. This shows the finiteness of $l$ such that $\Pi_K(l)$ is represented {by} $F(x,y)$.
\end{proof}
Since the $p$-factor of the above highly divisible sequences $H_l$ appears with regularity, we can replace $\Pi_K(l)$ with $H_l$ in Theorem \ref{fin}.
When $K=\mathbf Q$, the conjugacy class $C$ in Theorem \ref{fin} is equal to $\{1\}$ and $\deg F$ does not divide $[K:\mathbf Q]$. Therefore, we obtain the following corollary.
\begin{corollary}
\label{corbh}
Let $F(x,y)$ a polynomial in $\mathbf Z[x,y]$ whose irreducible factorization is \[F(x,y)=\prod_{j=1}^uF_j(x,y).\]
Assume that there exist positive integers $a$ and $b>1$ such that \[\bigcap_{i} \mathbf P(\mathcal C_{F_i})\supset\{p:\text{prime}~|~p\equiv a\mod b\}.\]
Then there exist only finitely many $l$ such that $H_l$ is represented {by} $F(x,y)$.
\end{corollary}
Taking $y=1$ in Corollary \ref{corbh},  we get the result of Berend and Harmse for reducible polynomials partially.
To explain their result, we introduce the natural density $d(S)$ for a subset $S$ of the set of all primes defined by
\[d(S) = \lim_{x\rightarrow\infty}\frac{\pi(x, S)}{\pi(x)},\]
where $\pi(x)$ is the number of primes $p\le x$ and $\pi(x, S)$ is the number of those belonging to $S$.
\begin{theorem}[Theorem 4.1. of \cite{bh}]
\label{q2}
Consider the equation
\begin{equation}
\label{equ2}
P(x)=H_l.
\end{equation}
Let $Q(x) \in\mathbf Z[x]$ be any factor (irreducible or not) of $P$. Denote by $S(Q) \subset P$ the set of all primes $p$ for which the congruence $Q(x)\equiv 0 \mod p$ has a solution. If $d(S(Q)) < \frac{\deg Q}{\deg P}$, then (\ref{equ2}) has only finitely many solutions.
\end{theorem}
The assumption $\bigcap_{i} \mathbf P(\mathcal C_{F_i})\supset\{p:\text{prime}~|~p\equiv a\mod b\}$ in Corollary \ref{corbh} leads to $d(S(F(x,1)))<1$. Thus, Theorem \ref{q2} implies Corollary \ref{corbh} with $y=1$.

If $\deg F$ is even then we can remove the assumption that $K/\mathbf Q$ is a Galois extension.
\begin{theorem}
\label{fin2}
Let $K$ be a number field and $F(x,y)$ a polynomial of even degree in $\mathbf Z[x,y]$ whose irreducible factorization is \[F(x,y)=\prod_{j=1}^uF_j(x,y).\]
Let $p$ be a prime whose Frobenius map $(p,K_P/\mathbf Q)$ has the cycle type $[f_1,\ldots,f_r]$ such that $G_p(m;K)$ is odd for some odd $m$ and $C$ the conjugacy class of $(p,K/\mathbf Q)$.
If there exist positive integers $a$ and $b>1$ such that $\mathbf P(C)\cap \bigcap_{i} \mathbf P(\mathcal C_{F_i})\supset\{p:\text{prime}~|~p\equiv a\mod b\}$, then there exist only finitely many $l$ such that $\Pi_K(l)$ is represented {by} $F(x,y)$.
\end{theorem}
\begin{proof}
We assume that $G_p(i;K)$ is odd for some prime $p$ and some odd $i$, and the Frobenius map $(p,K_P/\mathbf Q)$ has the cycle type $[f_1,\ldots,f_r]$ in $K$. Let $m=\min\{i~|~G_p(i;K)\text{ is odd}\}$.
For all odd $i< m$, the number $G_p(i;K)$ is even.
Let $a(n)$ be the number of ideals of $\mathcal O_K$ with $\mathfrak{Na}=n$
It follows from the Chinese Remainder Theorem that the function $a(n)$ satisfies the multiplicative property
\[a(mn)=a(m)a(n)\hspace{5mm}\text{ if }\gcd(m,n)=1.\]
{The ideals $\mathfrak a$ such that $\mathfrak{Na}=p^m$ is expressed by product of prime ideals $\mathfrak p$ with $\mathfrak{Np}=p^k (k\le m)$. If $\mathfrak a$ is expressed as $\mathfrak p_1\cdots\mathfrak p_s$ and the number of $\mathfrak p_t$ with $\mathfrak{Np}_t=p^i$ equals $a_i$, then we have $a_1+\cdots+ma_m=m$.
By considering the number of combinations with {repetition}, we get}
\[a(p^m)=\sum_{\substack{a_1+\cdots+ma_m=m\\a_i\ge0}}\prod_{i=1}^m\binom{G_p(i;K)+a_i-1}{a_i}.\]
Since $m$ is odd, there exists an odd $i$ such that $a_i$ is odd in each product. For this odd $i$ \[\binom{G_p(i;K)+a_i-1}{a_i}\] is even, since binomial coefficients {$\binom{e}{o}$, where $e$ is an even number and $o$ is an odd number,} are always even. Therefore
\[
a(p^m)=\sum_{\substack{a_1+\cdots+(m-1)a_{m-1}=m\\a_i\ge0}}\prod_{i=1}^m\binom{G_p(i;K)+a_i-1}{a_i}+G_p(m)=odd.\]
Accordingly, $\Pi_K(p^m)$ is not of the form in Theorem \ref{12}. Since $p^m$-factor does not appear in $(p^m,2p^m)$, for all $l\in(p^m,2p^m)$ the left hand side $\Pi_K(l)$ is not of the form in Theorem \ref{12}.

On the other hand, {a} Bertrand type estimate for primes in an arithmetic progression $bt+a$ leads {to the conclusion} that there exists $c> 0$ such that for $x > c$ there is a prime $p\in\{p:\text{prime}~|~p\equiv a\mod b\}$ with $p\in (x,2^{\frac1{m}}x)$.
Let $\mathfrak p$ be a prime ideal of $K$ corresponding to $C$ with $\mathfrak{Np}=p^m>\max\{c^m, (a_n\Delta_{F,mod})^m, (a_0\Delta_{F,mod})^m\}$.
Since $p>c$, there {exists} a prime $q$ satisfying $q=bt_q+a\in(p,2^{\frac1m}p)$. By the assumption, we have $q\in\mathbf P(C)$ and there exist $\mathfrak q$ lying above $q$ with $\mathfrak{Nq}=q^m\in (p^m,2p^m)$.

As above, $\Pi_K(l)$ is not of the form in Lemma \ref{12} for $q^m\le l<2q^m$ and there exists a prime ideal $\mathfrak q_1$ corresponding to $C$ with $\mathfrak{Nq}_1\in (q^m,2q^m)$.

By induction, $\Pi_K(l)$ does not satisfy the condition in Theorem \ref{12} for all $l\ge p^m$.
This shows the theorem.
\end{proof}
From Lemma \ref{pp}, we observe that for {a} number field $K$ there exists an effectively computable constant $c(2) > 0$ such that for
$l > \exp(c(2)k(\log D)^2)$ there is a prime $p$ {such} that $p^n$ divides $\Pi_K(l)$ but $p^{n+1}$ does not divide $\Pi_K(l)$.
Therefore, we obtain the following result.
\begin{theorem}
Let $K$ be a number field and $F(x,y)$ a polynomial in $\mathbf Z[x,y]$ whose irreducible factorization is \[F(x,y)=\prod_{j=1}^uF_j(x,y)^{e_m},\]
where $F_j$ are distinct irreducible polynomials.
If $\min\{e_1,\ldots,e_u\}>[K:\mathbf Q]$, then there exist only finitely many $l$ such that $\Pi_K(l)$ is represented {by} $F(x,y)$.
\end{theorem}
For special quadratic forms, we give a sufficient condition for existence of infinitely many solutions. We denote the set of primes which is inert in $\mathbf Q(\sqrt \Delta)$ by $P_\Delta$.
\begin{theorem}
Let $K$ be a number field with $n=[K:\mathbf Q]$ and $D_K$ its discriminant.
Let $F(x,y)=ax^2+bxy+cy^2$ be a positive definite quadratic form with fundamental modified discriminant $\Delta_{F,mod}$, where one of $a$ and $c$ is a prime number or $1$. We denote $P_{\Delta_{F},D_K}=P_{\Delta_{F}}\setminus\{p|D_K\}$.
We assume that the class number of $\mathbf Q(\sqrt{\Delta_{F}})$ equals $1$.  If for all $p\in P_{\Delta_{F,mod},D_K}$ and for odd $i$, $G_p(i;K)$ is even, then there exist infinitely many $l$ such that $\Pi_K(l)$ is represented {by} $F(x,y)$.
\end{theorem}
\begin{proof}
We assume for all $p\in P_{\Delta_{F},D_K}$ and odd $i$, $G_p(i;K)$ is even.
It suffices to show that the prime factorization of $\Pi_K(l)$ contains no prime $p\in P_{\Delta_{F}, D_K}$ raised to an odd power for infinitely many $l$.
From the multiplicative property of $a(n)$ we show $a(p^m)$ is even for all primes $p\in P_{\Delta_{F},D_K}$ and odd $m$ in the following.
As {in} the proof of Theorem \ref{fin2}, we get
\[a(p^m)=\sum_{\substack{a_1+\cdots+ma_m=m\\a_i\ge0}}\prod_{i=1}^m\binom{G_p(i;K)+a_i-1}{a_i}.\]
Now we assume $G_p(i;K)$ is even for all odd $i$. Since $m$ is odd, there exists an odd $i$ such that $a_i$ is odd in each product. As we remarked above, one of the binomial coefficients in the above product are even. Therefore, $a(p^m)$ is a sum of even numbers, $a(p^m)$ is also even.

If $G_p(i;K)$ is odd for some $p\in P_{\Delta_F}\setminus P_{\Delta_{F},D_K}$ and some odd $i$, we denote $m=\min\{i:\text{odd}~|~G_p(i,K) \text{ is odd}\}$. As we mentioned above, $a(p^m)$ is odd. Chebotarev's density theorem says that for any number fields $K$ there exist infinitely many primes splitting completely in $K$. Let $q$ be a prime splitting completely in $K$. Then we have $a(q^k)=\binom{n+k-1}{n-1}$. One can see easily that $\binom{n+k-1}{n-1}$ takes odd values infinitely many times and $a(p^mq^k)$ does.
Since $P_{\Delta_F}\setminus P_{\Delta_{F},D_K}$ is a finite set, $\Pi_K(l)$ satisfies the first and second conditions in Theorem \ref{1234} infinitely many times.
By assumption, the third condition in Theorem \ref{1234} is trivial.
This shows the theorem.
\end{proof}
\section{Some generalizations}
In the previous sections, we deal with two variables homogeneous polynomial. Naturally, we have an interest in the Brocard-Ramanujan problem for multi-variable homogeneous polynomial. In this section, we consider {polynomials in more variables}.

It is known that a positive integer $n$ is the sum of three squares of integers if and only if it is not of the form $4^l(8k + 7)$, where $k, l$ are non-negative integers. Using this criterion, we check that there are infinitely many $l$ such that $l!$ is represented {by} $x^2+y^2+z^2$.
Therefore, irreducibility of polynomials $f(x_1,\ldots,x_n)$ is not important for the finiteness of the solutions of $f(x_1,\ldots,x_n)=l!$.

In this paper, we consider the {equations involving} norm forms and factorial functions.
Let $\mathcal O$ be an order of number field $K$ and $\{\alpha_1,\ldots,\alpha_n\}$ be their basis over $\mathbf Z$.
Then the norm form $N_{\alpha_1,\ldots,\alpha_n}$ is defined by
\[N_{\alpha_1,\ldots,\alpha_n}(x_1,\ldots,x_n)=\prod_{\sigma\in{\rm Aut}(K/\mathbf Q)}\sigma\left(\sum_{i=1}^n\alpha_ix_i\right).\]
There exists the matrix $A$ converting the basis $\{\alpha_1,\ldots,\alpha_n\}$ to the basis $\{1,\alpha_2',\ldots,\alpha_n'\}$ of $\mathcal O$.
Also, since $A\in{\rm SL}_n(\mathbf Z)$, an integer $N$ is represented {by} $N_{\alpha_1,\ldots,\alpha_n}$ if and only if $N$ is also represented {by} $N_{1,\alpha_2',\ldots,\alpha_n'}$.
Therefore, it suffices to consider the case $\alpha_1=1$.

As {a} corollary of Theorem \ref{main1} we have
\begin{corollary}
For any norm form $N_{\alpha_1,\alpha_2}$ of quadratic fields, there exists only finitely many $l$ such that $H_l$ is represented {by} $N_{\alpha_1,\alpha_2}(x_1,x_2)$.
\end{corollary}
We generalize this corollary to all norm forms by following the proof of Theorem \ref{main1} as follows.
\begin{theorem}
\label{norm}
For any order $\mathcal O\not=\mathbf Z$ of a number field and their basis $\{\alpha_1,\ldots,\alpha_n\}$ over $\mathbf Z$, there exists only finitely many $l$ such that $H_l$ is represented {by} $N_{\alpha_1,\ldots,\alpha_n}(x_1,\ldots,x_n)$.
\end{theorem}
Using the following lemma, we can show Theorem \ref{norm} by the same argument with the proof of Theorem \ref{main1}.
\begin{lemma}
\label{normrep}
Let $\mathcal O$ be an order of a number field $K\neq\mathbf Q$ and $A=\{1,\alpha_2,\ldots,\alpha_n\}$ be a basis of $\mathcal O$ over $\mathbf Z$.
Let $F(x_1,x_2)$ be the homogeneous polynomial obtained by substituting $x_i=0$ for $i\ge3$ into $N_A(x_1,\ldots,x_n)$.
Let $N$ be an integer with \[N=p_1\cdots p_sq_1^{l_1}\cdots q_t^{l_t},\]
where $q_i$ are primes corresponding to a conjugacy class $C\in\mathcal C_F$ with\\ $\gcd(q_i,F(0,1)\Delta_{F})=1$ and $p_i$ are the other primes.
If $N$ is represented {by} $N_A$ then $l_i\ge2$ for all $i$.
\end{lemma}
\begin{remark}
By considering other intermediate fields $\mathbf Q(\alpha_i)$ of $K/\mathbf Q$, we {can} characterize the prime factorization of integers represented {by} $N_A$ more specifically.
\end{remark}

Next, we replace the right hand side of the equation $N_A(x)=l!$ with the Bhargava factorial. 
Bhargava introduced a generalization of the factorial function to
generalize classical results in $\mathbf Z$ to Dedekind domains and unify them \cite{bh97,bh00}. In this section, we consider the number of solutions for equations {involving} polynomials and the Bhargava factorial. Since the ordinary factorial $l!$ is one of examples of the Bhargava factorial, we regard this equation as one of the generalizations of the Brocard-Ramanujan problem. The Bhargava factorial is defined as follows.

Let $S$ be an infinite subset of $\mathbf Z$. First, we define $p$-ordering of $S$. A $p$-ordering of $S$ is any sequence $\{a_n\}$ of elements of $S$ that is formed as follows:
\begin{itemize}
    \item Choose any element $a_0\in S$;
    \item For $k\ge1$ choose an element $a_n\in S$ such that
\[v_p\left(\prod_{k=0}^{n-1}(a_n-a_k)\right)= \inf_{x\in S}v_p\left(\prod_{k=0}^{n-1}(x-a_k)\right),\]
where $v_p$ is the $p$-adic valuation defined by
$v_p(p^va)=v$ with an integer $a$ relatively prime to $p$.
\end{itemize}
For a $p$-ordering of $S$, we construct the $p$-sequence $\{v_p(n;S)\}$ as \[v_p(n;S)=v_p\left(\prod_{k=0}^{n-1}(a_n-a_k)\right).\]
It is known that the associated $p$-sequence of $S$ is independent of the choice of $p$-ordering of $S$ \cite{bh00}.

With these settings, we define the Bhargava factorial $l!_S$ by
\[l!_S=\prod_{p:\text{prime}}p^{v_{S}(l;p)}.\]
We give some examples of the Bhargava factorial.
When $S=\mathbf Z$, we can choose the natural ordering $0,1,2,3,\ldots$ as $p$-ordering of $l!_S$ for all primes $p$ and find $l!_S$ is the ordinary factorial $l!$. This is why we can regard the equation $P(x)=l!_S$ as one of the generalizations of the Brocard-Ramanujan problem $x^2-1=l!$.
Also, when $S(a,b)=\{an+b~|~n\in\mathbf Z\}$ for some $a,b\in\mathbf Z$ then $l!_{S(a,b)}=a^ll!$. Since the {arguments in} the proof of Theorem \ref{main1} also works for the equation $P(x)=l!_{S(a,b)}$, we obtain the finiteness of solutions $(x,l)$ for the equation $P(x)=l!_{S(a,b)}$.

In this section, we point out that we can generalize Luca's result by following his proof. Luca showed that the Oesterl\'e-Masser conjecture implies that the equation $P(x)=l!$ has only finitely many solutions $(x,l)$ \cite{lu02}. In the proof of this result, Luca used the facts that ${\rm rad}(l!)<4^l$ and the Stirling formula $\log l!\sim l\log l$ as $l\rightarrow\infty$ and he estimate ${\rm rad}(l!)$ and $l!$. Therefore, if we {can} estimate $l!_S$ and ${\rm rad}(l!_S)$, we can {argue as in} the proof of Luca's result. We {summarize} this more general form.
\begin{theorem}[cf. Luca. \cite{lu02}]
\label{summ}
Let $P(x)\in\mathbf Z[x]$ be a polynomial of $\deg P\ge2$ and $F(l)$ be a function satisfying ${\rm rad}(F(l))=o(F(l))$ as $l\rightarrow\infty$.
Then the Oesterl\'e-Masser conjecture implies that the equation $P(x)=F(l)$ has only finitely many solutions $(x,l)$.
\end{theorem}
We check that the Bhargava factorial for an infinite subset $S\subset \mathbf Z$ satisfies the condition of Theorem \ref{summ}. Since $l!|l!_S$, for all primes $p$, the $p$-adic valuation $v_p(l!_S)$ tends to infinity as $l\rightarrow\infty$. Therefore, as $l\rightarrow\infty$, ${\rm rad}(l!_S)=o(l!_S)$.
\begin{corollary}
Let $P(x)\in\mathbf Z[x]$ be a polynomial of $\deg P\ge2$ and $S$ be an infinite subset of $\mathbf Z$.
Then the Oesterl\'e-Masser conjecture implies that the equation $P(x)=l!_S$ has only finitely many solutions $(x,l)$.
\end{corollary}
For some special case, we can show the finiteness of solutions for the equation $P(x)=l!_S$ unconditionally.
\begin{theorem}
Let $N_{\alpha_1,\ldots,\alpha_n}(x_1,\ldots,x_n)$ be a norm form of number field $K\neq\mathbf Q$.
For a polynomial $f(x)=ax^2+bx+c$ with $(a,b,c)\in\mathbf Z^3-\{(0,0,c)~|~c\in\mathbf Z\}$ we denote $S=f(\mathbf Z)$. Then there exist only finitely many $l$ such that $l!_S$ is represented {by} $N_{\alpha_1,\ldots,\alpha_n}(x_1,\ldots,x_n)$.
\end{theorem}
\begin{proof}
As we remarked above, when $a=0$, we find that $l!_S=b^ll!$ and the same {argument as in} the proof of Theorem \ref{main1} also works for the equation $P(x)=l!_{S}$. Therefore, it suffices to show the case $a\neq0$. Let $p$ be an odd prime not dividing $a$. Then we have \[\#\{f(n)~|~n\in\mathbf Z\}=\frac{p+1}2\]
and we can choose an ordering $f(n_0),\ldots,f(n_{p-1}),\ldots$ satisfying the following three conditions:
\begin{enumerate}
\item $\{n_0,\ldots,n_{p-1}\}=[0,p-1]\cap\mathbf Z$;
\item $f'(n_0)\equiv0\mod p$;
\item For $0\le i<j\le\frac{p-1}2$, $f(n_i)\not\equiv f(n_j)\mod p$.
\end{enumerate}
This ordering forms a $p$-ordering of $S$ and we can estimate $v_p(l!_S)$ as
\begin{equation}\label{bhaest}v_p(l!_S)=\left\{\begin{array}{ll}
0&\text{if } 0\le l\le\frac{p-1}2,\\
1&\text{if } \frac{p+1}2\le l\le p-1,\\
2&\text{if } l=p.
\end{array}\right.\end{equation}
{Following} the notation in Lemma \ref{normrep}, we denote \[F(x_1,x_2)=N_{\alpha_1,\ldots,\alpha_n}(x_1,x_2,0,\ldots,0).\]
Lemma \ref{normrep} leads {to the conclusion} that if $N$ is represented {by} $N_{\alpha_1,\ldots,\alpha_n}$ and $p|N$ for odd prime $p$ corresponding to $C\in\mathcal C_F$ with $\gcd(q_i,F(0,1)\Delta_{F})$ $=1$, then $N$ is divisible by $p^2$ at least. In particular, $N_{\alpha_1,\ldots,\alpha_n}(x_1,\ldots,x_n)$ $=\frac{p+1}2!_S$ has no integer solution $(x_1,\ldots,x_n)$.
Moreover, since the second smallest $p$-factor appears at $l=p$ by estimate (\ref{bhaest}), $l!_S$ is not of the form in Lemma \ref{normrep} for $\frac{p+1}2\le l<p$, that is, there exists no $n$-tuple $(x_1,\ldots,x_n)\in\mathbf Z^n$ such that $N_{\alpha_1,\ldots,\alpha_n}(x_1,\ldots,x_n)=l!_S$ for $\frac{p+1}2\le l<p$.

As {in} the proof of Theorem \ref{main1}, for sufficient large prime $p$, there exists a prime $q$ corresponding to $C$ with $\frac{q+1}2\in (\frac{p+1}2,p)$ and $l!_S$ is not of the form in Lemma \ref{normrep} for $\frac{q+1}2\le l<q$. By induction, $l!_S$ is not of the form in Lemma \ref{normrep} for $\frac{p+1}2\le l$. This shows the finiteness of $l$ such that $l!_S$ is represented {by} $N_{\alpha_1,\ldots,\alpha_n}(x_1,\ldots,x_n)$.
\end{proof}
The case $\deg f\ge3$, it depends on the base field $K$. For example, when $f(x)=x^3$ then we find
\begin{equation}
\label{cubicbha}\#\{n^3\mod p~|~n\in\mathbf Z\}=\left\{\begin{array}{cl}
\frac{p+2}3&\text{if } p\equiv1\mod3,\\
p&\text{otherwise.}
\end{array}\right.\end{equation}
If $K/\mathbf Q$ is an abelian extension, then {there} exists a positive integer $D$ which characterizes the set of primes corresponding to a conjugacy class $C\subset{\rm Gal}(K/\mathbf Q)$. Therefore, for any norm form $N_A$ of $K$, we can show the finiteness of solutions for $N_A(x)=l!_S$.
On the other hand, if $K/\mathbf Q$ is not an abelian extension, then we cannot characterize the set of primes corresponding to a conjugacy class $C\subset{\rm Gal}(K/\mathbf Q)$ by any modulus and it is difficult to show the finiteness of solutions in general.
\section{acknowledgement}
The author deeply express their sincere gratitude to Professor M. Ram Murty and Professor Andrzej D{\k{a}}browski for fruitful discussions.
The author also deeply thanks Professor Kohji Matsumoto and Professor Masatoshi Suzuki for their precious advice.
This work was supported by Grant-in-Aid for JSPS Research Fellow (Grant Number: 19J10705).

\end{document}